\documentclass{amsart}

\usepackage{amssymb}
\usepackage{color}
\usepackage{array}

\usepackage[pdftex]{hyperref}
\hypersetup{
    colorlinks=true, 
    linktoc=all,    
    linkcolor=blue,  
}

\newcommand{\cal}{\mathcal}

\newcommand{\lb}{\left\{}
\newcommand{\rb}{\right\}}
\newcommand{\la}{\left<}
\newcommand{\ra}{\right>}

\newcommand{\ps}{\mathbb{P}}
\newcommand{\Q}{\mathbb{Q}}

\newcommand{\ka}{\kappa}

\newcommand{\si}{\sigma}

\newcommand{\vp}{\varphi}
\newcommand{\om}{\omega}

\newcommand{\res}{\upharpoonright}
\newcommand{\seq}{\subseteq}
\newcommand{\we}{\wedge}

\newcommand{\Add}{\operatorname{Add}}
\newcommand{\dom}{\operatorname{dom}}

\newcommand{\cl}{\operatorname{cl}}
\renewcommand{\res}{\operatorname{res}}

\newcommand{\es}{\emptyset}
\newcommand{\U}{\mathcal{U}}

\theoremstyle{definition}
\newtheorem{definition}{Definition}[section]

\newtheorem{example}[definition]{Example}

\newtheorem{question}[definition]{Question}

\theoremstyle{plain}

\newtheorem{theorem}[definition]{Theorem}

\newtheorem{lemma}[definition]{Lemma}

\newtheorem{corollary}[definition]{Corollary}

\newtheorem{remark}[definition]{Remark}

\keywords{Compactness, Menger, Rothberger, Countably Tight, Strongly Proper, Preservation of Topological Properties}

\subjclass[2010]{Primary: 03E05, 54D20, 91A44}

\begin{document}

\title{Preservation of Topological Properties by Strongly Proper Forcings}

\author{Thomas Gilton}\address[Gilton]{University of Pittsburgh
Department of Mathematics. The Dietrich School of 
Arts and Sciences, 301 Thackeray Hall,
Pittsburgh, PA, 15260 United States
} \email{tdg25@pitt.edu} \urladdr{https://sites.pitt.edu/~tdg25/}

\author{Jared Holshouser}\address[Holshouser]{Norwich University
Department of Mathematics. College of Arts and Sciences, 158 Harmon Dr, Northfield, VT 05663 United States
} \email{jholshou@norwich.edu} \urladdr{https://jaredholshouser.github.io/}

\date{\today}

\begin{abstract}
    In this paper we show that forcings which are strongly proper for stationarily many countable elementary submodels preserve each of the following properties of topological spaces: countably tight; Lindel{\"o}f; Rothberger; Menger; and a strategic version of Rothberger. This extends results from Dow's papers \cite{DowSubmodels} and \cite{DowTwoApplications}, as well as results of Iwasa (\cite{Iwasa}) and Kada (\cite{Kada}). 
\end{abstract}

\maketitle

\tableofcontents

\section{Introduction}

Our motivation for this paper is to understand classes of forcings that preserve topological properties of interest, such as countable tightness and the Lindel{\"o}f property. We first make some remarks about narrower classes of posets, and then we take a step back to motivate the investigation of broader classes of posets. It is known that, in general, countably closed posets needn't preserve these properties (\cite{ScheepersCountableTightness}). However, Dow has shown that posets of the form $\Add(\om,\om_1)\ast\dot{\Q}$, where $\Add(\om,\om_1)$ is the poset to add $\om_1$-many Cohen reals and $\dot{\Q}$ is a name for a countably closed poset, do in fact preserve countable tightness (Lemma 5.6 of \cite{DowSubmodels}) and do preserve the Lindel{\"o}f property (Lemma 3.3 of \cite{DowTwoApplications}).

Taking a step back, one particularly useful class of posets is the class of Proper posets, invented by Shelah (\cite{Shelah1},\cite{Shelah2}). The class of proper forcings subsumes both the class of countably closed and the class of ccc posets. Moreover, this class is iterable with countable supports, as Shelah proved, thereby generalizing the celebrated result of Solovay and Tennenbaum (\cite{SolovayTennenbaum}). Since countably closed forcings are proper, and since countably closed forcings do not in general preserve (for example) countable tightness and the Lindel{\"o}f property, we can't hope to prove that proper forcings in general preserve them. We remark that there are known results about the preservation of strengthenings of the Lindel{\"o}f property (\cite{Kada2}) by posets satisfying natural game-theoretic properties.

However, posets of the form $\Add(\om,\om_1)\ast\dot{\Q}$ (where $\dot{\Q}$ is forced to be countably closed), which do preserve these properties, are not just proper but satisfy the more stringent condition of being strongly proper (see Lemma 1.12 of \cite{GiltonNeeman} for a proof). Mitchell first explicitly isolated the idea of a strongly proper forcing (\cite{Mitchell_Ouch}). This class contains Cohen forcing to add any number of subsets of $\omega$; however no countably closed forcing (nor countably distributive forcing, for that matter) is strongly proper, since every (non-trivial) strongly proper forcing adds a Cohen real.

In light of these observations, we study preservation properties of topological spaces under strongly proper forcings, showing that this class is powerful enough to preserve countable tightness, the Lindel{\"o}f property, and many other covering properties. These generalize results from Dow, mentioned above, as well as results from Iwasa (\cite{Iwasa}) and Kada (\cite{Kada}), among others. We also note that Marun proved in his thesis, earlier and independently of us, that forcings which are strongly proper for stationarily many models preserve the Lindel{\"o}f property (\cite{Marun}).

We note, though, that we haven't exactly ``carved at the joints" here: there are non-strongly proper posets which do preserve many of these properties. For example, the forcing to add a random real is not strongly proper: this forcing does not add any unbounded reals, and hence does not add any Cohen real (as Cohen reals are unbounded). Nevertheless, the random real forcing does preserve many of the properties that we study here (see \cite{Kada}).

A remark on notation: ours is standard. One item that requires remark concerns names for ground model objects. When $x$ is a specific set in the ground model $V$ and when we are referring to $x$ in statements in the forcing language, we will either leave the symbol ``$x$" as is, or we will put a stylish hat on it, for example ``$\check{x}$." The use of the hat is to emphasize that this is a canonical name for the set $x$ in contexts in which some disambiguation could be helpful.

Finally, a remark on exposition: given that the material in this paper concerns set theory and topology, we have added expository background material about both set-theoretic as well as topological topics. The hope is that a reader who is more rooted in one of those fields, but who is interested in both, can profit from the material in the paper without having to consult too many other sources for background.

\section{Topological Definitions}

In this section, we present the relevant topological definitions, particularly for the reader who is more familiar with the set-theoretic side of things.

Some of the main themes of topology are compactness (as witnessed by open covers); convergence and clustering; and density. The study of selection principles takes these topics and expands on them with a particular focus on diagonalization processes. This section of preliminary definitions is broken up by these three themes and builds from standard definitions up to the selection principle versions.

First we discuss the variations on compactness, known as covering properties.
\begin{definition}
    Let $(X,\tau)$ be a topological space.
    \begin{itemize}
        \item $X$ is \textbf{compact} if for every open cover $\mathcal U$ of $X$, there is a finite subset $\mathcal F \subseteq \mathcal U$ so that $X = \bigcup \mathcal F$.
        \item $X$ is \textbf{Lindel\"{o}f} if for every open cover $\mathcal U$ of $X$, there is a countable subset $\mathcal V \subseteq \mathcal U$ so that $X = \bigcup \mathcal V$.
        \item $X$ is \textbf{Rothberger} if for every sequence $\langle \mathcal U_n : n \in \omega \rangle$ of open covers of $X$, there are open sets $U_n \in \mathcal U_n$ so that $X = \bigcup_n U_n$.
        \item $X$ is \textbf{Menger} if for every sequence $\langle \mathcal U_n : n \in \omega \rangle$ of open covers of $X$, there are finite subsets $\mathcal F_n \subset \mathcal U_n$ so that $X = \bigcup_n \bigcup \mathcal F_n$.
    \end{itemize}
\end{definition}

Conveniently, the notion of Rothberger is slightly stronger than it looks:

\begin{lemma}\label{lemma:StrongBerger}
    A space $(X,\tau)$ is Rothberger if and only if for every sequence $\langle \mathcal U_n : n \in \omega \rangle$ of open covers of $X$, there are open sets $U_n \in \mathcal U_n$ so that
    for every $x \in X$, there are infinitely many $n \in \omega$ so that $x \in U_n$.
\end{lemma}

Likewise, the Menger property is stronger than it initially appears:

\begin{lemma}\label{lemma:StrongMenger}
     A space $(X,\tau)$ is Menger if and only if for every sequence $\langle \mathcal U_n : n \in \omega \rangle$ of open covers of $X$, there exists a sequence $\la\cal{F}_n:n\in\om\ra$ so that 
     \begin{itemize}
         \item each $\cal{F}_n$ is a finite subset of $\cal{U}_n$; and
         \item for every $x \in X$, there are infinitely many $n \in \omega$ so that for some $W\in\cal{F}_n$, $x \in W$.
     \end{itemize}
\end{lemma}

There are variations of these covering properties that naturally arise in $C_p$-theory, the study of the space of continuous real-valued functions on $X$.

\begin{definition}
    Let $(X,\tau)$ be a topological space.
    \begin{itemize}
        \item $\mathcal U$ is an $\omega$-\textbf{cover} of $X$ if whenever $F \subseteq X$ is finite, there is a $U \in \mathcal U$ so that $F \subseteq U$.
        \item $X$ is $\omega$-\textbf{Lindel\"{o}f} if for every $\omega$-cover $\mathcal U$ of $X$, there is a countable subset $\mathcal V \subseteq \mathcal U$ so that $\mathcal V$ is an $\omega$-cover of $X$.
        \item $X$ is $\omega$-\textbf{Rothberger} if for every sequence $\langle \mathcal U_n : n \in \omega \rangle$ of $\omega$-covers  of $X$, there are open sets $U_n \in \mathcal U_n$ so that $\{U_n : n \in \omega\}$ is an $\omega$-cover of $X$.
        \item $X$ is $\omega$-\textbf{Menger} if for every sequence $\langle \mathcal U_n : n \in \omega \rangle$ of $\omega$-covers of $X$, there are finite subsets $\mathcal F_n \subset \mathcal U_n$ so that $\{\bigcup \mathcal F_n : n \in \omega\}$ is an $\omega$-cover of $X$.
    \end{itemize}
\end{definition}

\begin{remark}\label{remark:omegaversion}
    It is worth noting that for the properties listed above, a space $X$ satisfies the ``$\omega$-version" iff each of its finite powers satisfies the ``ordinary version." More explicitly, $X$ is $\omega$-Lindel\"{o}f if and only if $X^n$ is Lindel\"{o}f for all $n$ \cite{GerlitzNagy}, $X$ is $\omega$-Menger if and only if $X^n$ is Menger for all $n$ \cite{JustMillerScheepers}, and $X$ is $\omega$-Rothberger if and only if $X^n$ is Rothberger for all $n$ \cite{Sakai}.
\end{remark}

For the sake of the reader who may not be familar with these definitions, we now take a moment to show how to separate many of them. Note first that Rothberger implies Menger which in turn implies Lindel\"{o}f and also that compact implies Lindel\"{o}f. Similarly, $\omega$-Rothberger implies $\omega$-Menger which in turn implies $\omega$-Lindel\"{o}f. Finally, the ``$\omega$ version" of each property implies the ordinary version. To separate the properties, we have the following counterexamples.
\begin{itemize}
    \item The Sorgenfrey line is Lindel\"{o}f, but it is not Menger \cite{Scheepers2003}. Neither is it $\om$-Lindel{\"o}f.
    \item The real line with the standard topology is Menger, and in fact, $\om$-Menger. But is not Rothberger (and hence not $\om$-Rothberger). The fact that $\mathbb R$ is Menger comes from writing $\mathbb R = \bigcup_{n \in \omega} [-n,n]$. The fact that it is not Rothberger comes from the sequence of covers given by
    \[
    \mathcal U_n = \{(a - 2^{-n},a + 2^{-n} : a \in \mathbb R\}
    \]
    for each $n$.
    \item Any countable space is Rothberger. Uncountable examples include the Fortissimo space of an uncountable discrete space: Let $X$ be uncountable and discrete and $x^* \notin X$. Topologize $X \cup\{x^*\}$ by letting $U \subseteq X \cup\{x^*\}$ be open if it does not contain $x^*$ or if its complement is countable. See the entry for this space at the Pi-Base [\href{https://topology.pi-base.org/spaces/S000022}{S22}] for more details. Luzin subsets of the real line are also Rothberger \cite{Rothberger}.
    \item The irrational numbers are $\omega$-Lindel\"{o}f, but they are not Menger (and perforce not $\omega$-Menger). For a reference, see the entry at the Pi-Base [\href{https://topology.pi-base.org/spaces/S000028}{S28}].
    \item It is hard to find examples of spaces that are Menger but not $\omega$-Menger. In fact, in \cite{Zdomskyy}, it is shown that it is consistent with ZFC that for all metrizable spaces, the Menger and $\omega$-Menger properties are equivalent.
    \item Under CH, there are Luzin subsets of the real line that are Rothberger but not $\omega$-Rothberger. On the other hand, in the Laver model, every Rothberger subset of the real line is countable, and thus $\omega$-Rothberger \cite{Zdomskyy2}\cite{Laver}\cite{Rothberger2}.
\end{itemize}

Now we turn to reviewing convergence and clustering properties, the study of which helps to address the question ``How do we characterize the closure of a subset?" We will see below that there are ways of approaching this both through clustering and through convergence. 
\begin{definition}
    Let $(X,\tau)$ be a topological space. For $A \subseteq X$, we will use $\mbox{cl}(A)$ to denote the closure of $A$.
    \begin{itemize}
        \item $X$ has \textbf{countable tightness at $x$} if for every $A \subseteq X$ with $x \in \mbox{cl}(A)$, there is a countable set $B \subseteq A$ so that $x \in \mbox{cl}(B)$. $X$ has \textbf{countable tightness} if it has countable tightness at $x$ for all $x \in X$.
        \item $X$ is \textbf{Fr\'{e}chet-Urysohn at $x$} if for every $A \subseteq X$ with $x \in \mbox{cl}(A)$, there is a sequence $\langle x_n : n \in \omega \rangle$ from $A$ so that $x_n \to x$. $X$ is \textbf{Fr\'{e}chet-Urysohn} if it is Fr\'{e}chet-Urysohn at $x$ for all $x \in X$.
        \item $X$ has \textbf{strong countable fan tightness at $x$} if for every sequence of subsets $\langle A_n : n \in \omega \rangle$ with $x \in \mbox{cl}(A_n)$ for all $n$, we can find an $x_n \in A_n$ for each $n$ so that $x \in \mbox{cl}(\{x_n : n \in \omega\})$. $X$ has \textbf{strong countable fan tightness} if it has strong countable fan tightness at $x$ for all $x \in X$.
        \item $X$ has \textbf{countable fan tightness at $x$} if for every sequence of subsets $\langle A_n : n \in \omega \rangle$ with $x \in \mbox{cl}(A_n)$ for all $n$, we can find finite subsets $F_n \subseteq A_n$ for each $n$ so that $x \in \mbox{cl}(\bigcup_n F_n)$. $X$ has \textbf{countable fan tightness} if it has countable fan tightness at $x$ for all $x \in X$.
        \item $X$ is \textbf{strongly Fr\'{e}chet-Urysohn at $x$} if for every sequence of subsets $\langle A_n : n \in \omega \rangle$ with $x \in \mbox{cl}(A_n)$ for all $n$, we can find an $x_n \in A_n$ for each $n$ so that $x_n \to x$. $X$ is \textbf{strongly Fr\'{e}chet-Urysohn} if it is strongly Fr\'{e}chet-Urysohn at $x$ for all $x \in X$.
    \end{itemize}
\end{definition}

Finally, we have a few remarks about density.

\begin{definition}
    Let $(X,\tau)$ be a topological space.
    \begin{itemize}
        \item $X$ is \textbf{separable} if there is a countable $D \subseteq X$ so that $X = \mbox{cl}(D)$.
        \item $X$ is \textbf{selectively separable} if for every sequence $\langle D_n : n \in \omega \rangle$ of dense subsets of $X$, we can find $x_n \in D_n$ so that $\{x_n : n \in \omega\}$ is dense.
    \end{itemize}
\end{definition}

We take another moment here to discuss how these properties relate to each other. The first chain of implications is that strong countable fan tightness implies countable fan tightness which in turn implies countable tightness. Note also that selectively separable implies separable. It is also true that the strong Fr\'{e}chet-Urysohn property implies the Fr\'{e}chet-Urysohn property which implies countable tightness. We can separate the properties out with some counter examples.
\begin{itemize}
    \item The Arens-Fort space is countably tight but not Fr\'{e}chet-Urysohn. This space is defined to be $\omega^2$ where a set is open if it either does not contain $(0,0)$ or it contains all but a finite number of points in all but a finite number of columns. It is countably tight because it is countable, but it fails to be Fr\'{e}chet-Urysohn because no sequence converges to $(0,0)$.
    \item The space $C_p(\mathbb R)$ of continuous functions on $\mathbb R$ with the topology of point-wise convergence, is countably tight at all points (because the real line is $\omega$-Lindel\"{o}f), but it does not have strong countable fan tightness at any point (because the real line is not $\omega$-Rothberger, see \cite{ClontzHolshouser}).
    \item If $X$ is the irrational numbers, then $C_p(X)$ is separable, but as the irrationals are not $\omega$-Menger, $C_p(X)$ is not selectively separable (see Theorem 35 of \cite{Scheepers1999}).
\end{itemize}

The unified way in which these definitions developed is summarized by the concept of a selection principle.

\begin{definition}
    Let $\mathcal A$ and $\mathcal B$ be collections of sets.
    \begin{itemize}
        \item We say that $S_1(\mathcal A, \mathcal B)$ is true if for every sequence $\langle A_n : n \in \omega \rangle$ of sets from $\mathcal A$, we can find $x_n \in A_n$ so that $\{x_n : n \in \omega\} \in \mathcal B$. This is called a \textbf{single selection principle}.
        \item We say that $S_{\mbox{fin}}(\mathcal A, \mathcal B)$ is true if for every sequence $\langle A_n : n \in \omega \rangle$ of sets from $\mathcal A$, we can find finite subsets $F_n \subseteq A_n$ so that $\bigcup_n F_n \in \mathcal B$. This is called a \textbf{finite selection principle}.
    \end{itemize}
\end{definition}

If we allow ourselves some notation to compactly describe the topological objects of interest, we can re-frame many of the previous definitions as selection principles. Let
\begin{itemize}
    \item $\mathcal O(X)$ denote the open covers of $X$,
    \item $\Omega(X)$ denote the $\omega$-covers of $X$,
    \item for $x \in X$, $\Omega_{X,x}$ denote $\{A \subseteq X : x \in \mbox{cl}(A)\}$,
    \item for $x \in X$, $\Gamma_{X,x}$ denote the set of sequences from $X$ that converge to $x$, and
    \item $\mathcal D_X$ denote the dense subsets of $X$.
\end{itemize}

Here is a sample of the equivalences we obtain with this notation.
\begin{itemize}
    \item $X$ is Rothberger if and only if $S_1(\mathcal O(X), \mathcal O(X))$ is true.
    \item $X$ is $\omega$-Menger if and only if $S_{\mbox{fin}}(\Omega(X), \Omega(X))$ is true.
    \item $X$ has strong countable fan tightness at $x$ if $S_1(\Omega_{X,x},\Omega_{X,x})$ is true.
    \item $X$ is selectively separable if $S_1(\mathcal D_X, \mathcal D_X)$ is true.
\end{itemize}

For more details on selection principles and relevant references, see \cite{Scheepers1996} \cite{Scheepers2001}, \cite{Scheepers2003}, and \cite{Kocinac}.

\section{A Brief Overview of Strongly Proper Forcings}

In this section, we will present the basics of strongly proper forcings. The reader who is already familiar with this material need only consult Definition \ref{def:ResFunction} for some notation that we will use throughout the paper.

For the next three definitions, fix a large enough regular cardinal $\theta\geq\om_2$, a countable elementary submodel $M\prec H(\theta)$, and a poset $\ps\in M$.

\begin{definition}
    Let $q\in\ps$ be a condition. A condition $\bar{q}\in M\cap\ps$ is a \emph{residue} of $q$ to $M$ if for all $s\leq\bar{q}$ with $s\in M$, $s$ is compatible with $q$.
\end{definition}

In other words, every extension of $\bar{q}$ within $M$ is compatible with $q$.

\begin{definition}
    A condition $p\in\ps$ is an $(M,\ps)$-\emph{strongly generic condition} if every $q\leq p$ has a residue $\bar{q}$ to $M$.
\end{definition}

The next layer in the definitions specifies when the poset $\ps$ is strongly generic for the model $M$; the definition roughly asserts that there are plenty of $(M,\ps)$-strongly generic conditions.

\begin{definition}
    $\ps$ is \emph{strongly proper} for $M$ if every $s\in M\cap\ps$ can be extended to an $(M,\ps)$-strongly generic condition.
\end{definition}

It will be helpful to have a function which, for each $(M,\ps)$-strongly generic condition $q$, chooses a residue of $q$ to $M$. The next definition addresses this.

\begin{definition}\label{def:ResFunction}
    Suppose that $\ps$ is strongly proper for a countable $M\prec H(\theta)$. An $(M,\ps)$-\emph{residue function} is a function $\res_M$ whose domain is the set of $(M,\ps)$-strongly generic conditions and which satisfies that for any such condition $p$, $\res_M(p)$ is a residue of $p$ to $M$.
\end{definition}

Finally, we specify when $\ps$ is strongly generic for a set of models:

\begin{definition}
    Let $\cal{S}\seq [H(\theta)]^{\aleph_0}$ be a set of countable elementary submodels of $H(\theta)$. $\ps$ is \emph{strongly generic} for $\cal{S}$ if for all $M\in\cal{S}$, $\ps$ is strongly generic for $M$.
\end{definition}

In practice, the set $\cal{S}$ is large in the sense of being \emph{at least} stationary in $[H(\theta)]^{\aleph_0}$. Rather than reviewing the concepts of generalized stationarity, we provide the following lemma giving the consequence of stationarity that we use in practice:

\begin{lemma}
    Suppose that $\cal{S}\seq [H(\theta)]^{\aleph_0}$ is a stationary set of countable elementary submodels of $H(\theta)$. Then for any $x\in [H(\theta)]^{\leq\aleph_0}$, there is an $M\in\cal{S}$ so that $x\in M$.
\end{lemma}

Note in the conclusion of the previous lemma that since $x\in M$ and $x$ is countable, we in fact have $x\subset M$.

\begin{remark}
    Throughout this paper we will be working with an arbitrary poset $\ps$ which is strongly proper for a stationary $\cal{S}\seq [H(\theta)]^{\aleph_0}$. This is weaker than the usual definition of strongly proper. The usual definition says that $\ps$ is a strongly proper poset if it is strongly proper for a \emph{club} of countable elementary submodels of $H(\theta)$ for all large enough regular $\theta$.
\end{remark}

The next item shows that the value of $\theta$ doesn't play a significant role; the proof of the lemma relies on standard facts about generalized stationarity.

\begin{lemma}
    Suppose that $\theta<\theta^*$ are regular and $\ps\in H(\theta)$ a poset. Then $\ps$ is strongly proper for a stationary $\cal{S}\seq[H(\theta)]^{\aleph_0}$ iff $\ps$ is strongly proper for a stationary $\cal{S}^*\seq[H(\theta^*)]^{\aleph_0}$.
\end{lemma}

We now give one example of a strongly proper forcing, for the sake of exposition.

\begin{definition}
    Let $\ka$ be a regular cardinal. $\Add(\om,\ka)$ is the poset to add $\ka$-many Cohen reals. Conditions are finite partial functions $p:\ka\times\om\rightharpoonup 2$. $q$ extends $p$, denoted $q\leq p$, if $p\seq q$.
\end{definition}

\begin{example}
    Let $\theta>\ka$ both be regular. Then $\Add(\om,\ka)$ is strongly proper for any countable $M\prec H(\theta)$ with $\ka\in M$.

    Indeed, let $M\prec H(\theta)$ be given so that $\ka\in M$, from which we conclude that the poset $\Add(\om,\ka)$ is in $M$. We claim that every condition $p\in\Add(\om,\ka)$ is a strongly $(M,\ps)$-generic condition. Fix any condition $p$. Let $\bar{p}$ be the restriction of $p$ to $\dom(p)\cap M$. Then $\bar{p}$ is a finite subset of $M$ and hence an element of $M$ (recall that a residue must be a member of the model). Any $s\leq\bar{p}$ with $s\in M$ is compatible with $p$ since $s$ and $p$ agree on their common domain, namely, $\dom(p)\cap M=\dom(\bar{p})$. 
\end{example}

Strongly proper forcings form a much wider class than Cohen forcing to add subsets of $\om$. The next lemma (see Lemma 1.12 of \cite{GiltonNeeman} for a proof) gives an example of a different kind of forcing which is strongly proper. 

\begin{lemma}\label{lemma:stronglypropertwostep}
    Let $\dot{\Q}$ be an $\Add(\om,\om_1)$-name for a countably closed poset, and define $\ps:=\Add(\om,\om_1)\ast\dot{\Q}$. Then for any large enough regular $\theta$ and any countable $M\prec H(\theta)$ with $\dot{\Q}\in M$, $\ps$ is strongly proper for $M$.
\end{lemma}

However, the class of strongly proper posets does \emph{not} include any countably-closed posets (and hence it is a smaller class than the class of proper posets). The next item, a standard fact about strongly proper forcings, explains why:

\begin{lemma}
    Suppose that $\theta$ is a large enough regular cardinal and that $\ps$ is strongly proper for a stationary $\cal{S}\seq[H(\theta)]^{\aleph_0}$. Then forcing with $\ps$ adds a Cohen real over $V$.
\end{lemma}
\begin{proof}
    Let $p\in\ps$ be arbitrary. Let $M\in\cal{S}$ so that $p\in M$. Extend $p$ to an $(M,\ps)$-strongly generic condition $q$. Then $q$ forces that $\dot{G}\cap (M\cap\ps)$ is $V$-generic over $M\cap\ps$. But $M\cap\ps$ is a non-trivial, countably infinite poset. Hence $M\cap\ps$ contains a dense subset isomorphic to $\Add(\om,1)$. From this it follows that a $V$-generic for $M\cap\ps$ adds a $V$-generic over $\Add(\om,1)$.
\end{proof}

\section{Preservation of Countably Tight spaces and of Lindel{\"o}f spaces}

In this section, we introduce our first round of preservation theorems. We show that forcings which are strongly proper for stationarily-many models preserve the countable tightness of a space; we also show that such forcings preserve when a space is Lindel{\"o}f. In the next section, we will address the preservation of Menger and Rothberger spaces, and we will need to introduce some additional technology for those arguments. We begin by clarifying what it means for a forcing to preserve a topological property.

Fix a topological space $(X,\tau)$ and a poset $\ps$ in the ground model $V$. In the extension of $V$ by a $V$-generic $G$, we consider the set $X$ with the topology $\tau^G$ which has $\tau$ as a basis. Intuitively speaking, the information from the generic filter allows us to combine old things (open sets in $\tau$) in new ways to create additional open sets.

We say that $\ps$ preserves some topological property $\Phi$ of $(X,\tau)$ if every condition in $\ps$ forces that $(X,\tau^{\dot{G}})$ has property $\Phi$.

\begin{example}
    For a simple (counter)example, consider the unit interval $[0,1]$ from the ground model, and let $\ps$ be a poset that adds a new subset of $\om$, and hence a new real. Note that $\ps$ then adds a new real between 0 and 1. Then $[0,1]^V$ is no longer compact in $V[G]$. Since $\ps$ adds a new real between 0 and 1, we can use rational open intervals from $V$ to define, in $V[G]$, an open cover of $[0,1]^V$ with no finite subcover.
\end{example}

\begin{remark}
    When checking that a property $\Phi$ of a space holds, we often need only check the property with respect to basic open sets. In the context of proving preservation theorems, this allows us to simplify the discussion by working with names for open sets in the ground model, since the ground model topology $\tau$ is a basis for the topology on $X$ in $V[G]$.
\end{remark}

\begin{definition}\label{def:trace}
    Let $\ps$ be a poset, $(X,\tau)$ a topological space and $\dot{\cal{U}}$ a $\ps$-name for an open cover of $(X,\tau)$ forced to be a subset of $\tau$. Fix $p\in\ps$. We define the \emph{trace of $\dot{\cal{U}}$ below $p$}, denoted $\cal{U}_p$, to be the following set:
    $$
    \cal{U}_p:=\lb W\in\tau:(\exists q\leq p)\;[q\Vdash \check{W}\in\dot{\cal{U}}]\rb.
    $$
\end{definition}

\begin{lemma}\label{lemma:trace}
    Suppose that $\dot{\cal{U}}$ is a $\ps$-name for an open cover of $(X,\tau)$ forced to be a subset of $\tau$. Let $p\in\ps$ be arbitrary. Then $\cal{U}_p$ is an open cover of $(X,\tau)$.
\end{lemma}
\begin{proof}
    Fix $x\in X$. Let $G$ be an arbitrary $V$-generic filter over $\ps$ with $p\in G$. Then $\cal{U}:=\dot{\cal{U}}[G]$ is an open cover of $(X,\tau^G)$ consisting of sets in $\tau$. Let $W\in\cal{U}$ so that $x\in W$. Since $W\in\tau$ is a set in the ground model, we may find a condition $q\leq p$ with $q\in G$ so that
    $$
    q\Vdash \check{W}\in\dot{\cal{U}}.
    $$
    Then $q$ witnesses that $W\in\cal{U}_p$. Since $x\in W$, this completes the proof.
\end{proof}

\begin{lemma}\label{lemma:lovely}
Suppose that $(X,\tau)$ is a Lindel{\"o}f space. Let $\ps$ be a poset, $\theta$ a large enough regular cardinal, and $M\prec H(\theta)$ countable containing  $\ps$ and $(X,\tau)$ as elements. Let $\cal{\dot{\U}}$ be a $\ps$-name for an open cover of $X$ with $\cal{\dot{\U}}\in M$ so that $\,\dot{\cal{U}}$ is forced to be a subset of $\tau$.

Then for all $x\in X$ and all $s\in M\cap\ps$, there exist $t\leq_{M\cap\ps} s$ and $W\in M\cap\tau$ so that
$$
t\Vdash x\in \check{W}\in\dot{\cal{U}}.
$$
\end{lemma}
\begin{proof}
    Fix $s\in M\cap\ps$. Let $\cal{U}_s$ be the trace of the name $\dot{\cal{U}}$ below $s$, as defined in Definition \ref{def:trace}. By Lemma \ref{lemma:trace}, $\cal{U}_{s}$ is an open cover of $(X,\tau)$. Additionally, $\cal{U}_{s}$ is a member of $M$ since it is definable in $H(\theta)$ by parameters in $M$. By applying the fact that $(X,\tau)$ is Lindel{\"o}f and the elementarity of $M$, we may find a countable $\cal{B}_{s}\seq\cal{U}_{s}$ so that
    \begin{itemize}
        \item $\cal{B}_{s}$ is an open cover of $(X,\tau)$ and
        \item $\cal{B}_{s}\in M$.
    \end{itemize}
    
    Since $\cal{B}_{s}$ is countable and a member of $M$, it is a subset of $M$. Since $\cal{B}_{s}$ is an open cover, we may find $W\in\cal{B}_{s}$ with $x\in W$. However, $W\in\cal{U}_{s}$ and $W\in M$. Therefore $M$ contains a witness that $W\in\cal{U}_{s}$; that is, there is a $t\leq_{M\cap\ps}s$ so that $t\Vdash W\in\dot{\cal{U}}$. This completes the proof of the lemma.
\end{proof}

    For the reader who may not be familiar with strongly proper forcings, it is worth making a remark here about the structure of the following proofs. When showing that a certain property of $(X,\tau)$ is preserved by $\ps$, we will begin with an arbitrary condition $p\in\ps$ with the goal of extending $p$ to a condition which forces the desired property (hence $\ps$ itself forces the preservation of the desired property). We will find the condition $q$ by first placing $p$, as well as other parameters of interest, inside of some countable elementary submodel $M$. An $(M,\ps)$-strongly generic condition $q$ extending $p$ is then the desired condition. To see this, we proceed to argue that no extension, say $r$, of $q$ can force a specific counterexample. This will involve taking a residue $\res_M(r)$ of $r$ to $M$ and working with the elementarity of $M$.

\begin{theorem}\label{LindelofPreserved}
    Let $\theta$ be a large enough regular cardinal, $(X,\tau)\in H(\theta)$ a Lindel{\"o}f space, and $\ps\in H(\theta)$ a poset. Suppose that the set of countable $M\prec H(\theta)$ so that $\ps$ is strongly proper for $M$ is stationary in $[H(\theta)]^{\aleph_0}$. Then $\ps$ forces that $(X,\tau)$ is Lindel{\"o}f.
\end{theorem}
\begin{proof}
    Fix a condition $p\in\ps$ and a $\ps$-name $\dot{\cal{U}}$ for an open cover of $X$ consisting of elements of $\tau$. We will find a condition $q\leq p$ and a $\ps$-name $\dot{\cal{U}}_0$ so that
    $$
    q\Vdash\dot{\cal{U}}_0\in\left[\dot{\cal{U}}\right]^{\aleph_0}\text{ and }\;\dot{\cal{U}}_0\text{ covers }X.
    $$

    By the stationarity assumed in the statement of the theorem, we may find a countable $M\prec H(\theta)$ so that $\lb (X,\tau),\ps,p\rb\subset M$ and so that every condition in $M\cap\ps$ has an extension to an $(M,\ps)$-strongly generic condition. In particular, we may find $q\leq p$ so that $q$ is $(M,\ps)$-strongly generic. Let $\res_M$ be a residue function for $M$ (see Definition \ref{def:ResFunction}).
    
    Define $\dot{\cal{U}}_0$ by taking an arbitrary $V$-generic $G$ over $\ps$ and setting $\dot{\cal{U}}_0[G]=M\cap\dot{\cal{U}}[G]$. Note that $\dot{\cal{U}}_0$ is forced to be countable since it is forced to be a subset of $M$, which itself is countable. We claim that
    $$
    q\Vdash \dot{\cal{U}}_0\text{ covers }X.
    $$

    To prove this claim, it suffices to show that no $r\leq q$ can force a specific counterexample. In other words, it suffices to argue that for all $r\leq q$ and all $x\in X$, there exist an open set $W\in M\cap\tau$ and a condition $r^*\leq r$ so that
    $$
    r^*\Vdash x\in \check{W}\we \check{W}\in\dot{\cal{U}};
    $$
    note that such an $r^*$ then forces $\check{W}\in\dot{\cal{U}}_0$.
    Thus fix such $r$ and $x$. Since $q$ is strongly $(M,\ps)$-generic and $r\leq q$, we know that $\bar{r}:=\res_M(r)$ is defined. By Lemma \ref{lemma:lovely}, we may find a $t\leq_{M\cap\ps}\bar{r}$ and an open $W\in M\cap\tau$ so that
    $$
    t\Vdash x\in\check{W}\in\dot{\cal{U}}.
    $$
    Since $t\in M$ and $t$ extends $\bar{r}$, we know that $t$ and $r$ are compatible. Let $r^*$ be a condition extending both, completing the proof.
    \end{proof}

We obtain the preservation of the $\omega$-Lindel\"{o}f property as a corollary of this theorem.

\begin{corollary}\label{OmegaLindelof}
        Let $\theta$ be a large enough regular cardinal, $(X,\tau)\in H(\theta)$ an $\omega$-Lindel{\"o}f space, and $\ps\in H(\theta)$ a poset. Suppose that the set of countable $M\prec H(\theta)$ so that $\ps$ is strongly proper for $M$ is stationary in $[H(\theta)]^{\aleph_0}$. Then $\ps$ forces that $(X,\tau)$ is $\omega$-Lindel{\"o}f.
\end{corollary}
\begin{proof}
    Recall that a space is $\omega$-Lindel\"{o}f if and only if $X^n$ is Lindel\"{o}f for all $n$. Assuming $X$ is $\omega$-Lindel\"{o}f, Theorem \ref{LindelofPreserved} implies that $X^n$ is forced to be Lindel\"{o}f for all $n$. This is turn suffices to show that $X$ is forced to be $\omega$-Lindel\"{o}f.
\end{proof}

Now we address the preservation of countable tightness of spaces by posets which are strongly proper for stationarily many models. Our theorem generalizes Lemma 5.6 of \cite{DowSubmodels}, in which Dow showed that any poset of the form $\Add(\om,\om_1)\ast\dot{\Q}$, where $\dot{\Q}$ is forced to be countably closed, preserves countable tightness. Such posets, as stated in Lemma \ref{lemma:stronglypropertwostep}, are strongly proper. We also recall that a countably closed poset needn't preserve the countable tightness of a space.

\begin{theorem}
    Let $\theta$ be a large enough regular cardinal, $(X,\tau)\in H(\theta)$ a countably tight space, and $\ps\in H(\theta)$ a poset. Suppose that the set of countable $M\prec H(\theta)$ so that $\ps$ is strongly proper for $M$ is stationary in $[H(\theta)]^{\aleph_0}$. Then $\ps$ forces that $(X,\tau)$ is countably tight.
\end{theorem}
\begin{proof}
    Fix a condition $p\in\ps$, a point $x^*\in X$, and a $\ps$-name $\dot{A}$ so that 
    $$
    \Vdash_\ps x^*\in\cl(\dot{A}),
    $$
    i.e., so that this is forced by the whole poset. We will find a condition $q\leq p$ and a $\ps$-name $\dot{B}$ so that
    $$
    q\Vdash\dot{B}\in [\dot{A}]^{\aleph_0}\we x^*\in\cl(\dot{B}).
    $$

    By the stationarity assumed in the statement of the theorem, we may find a countable $M\prec H(\theta)$ so that $\lb (X,\tau),\ps,p,x^*,\dot{A}\rb\subset M$ and so that every condition in $M\cap\ps$ has an extension to an $(M,\ps)$-strongly generic condition. In particular, we may find $q\leq p$ so that $q$ is $(M,\ps)$-strongly generic. Let $\res_M$ be a residue function for $M$.

    Since $M$ is countable, $M\cap\dot{A}$ is forced to be a countable subset of $\dot{A}$. Taking $M\cap\dot{A}$ to be the name $\dot{B}$ above, it suffices to show that $q\Vdash x^*\in\cl(M\cap\dot{A})$. Since $\tau$ (the topology in the ground model) is a basis for the topology in the extension, it suffices to verify the definition of $x^*$ being in the closure of $M\cap\dot{A}$ when applied to open sets from the ground model.

    In light of this, fix a condition $r\leq q$ and an open $U\in \tau$ with $x^*\in U$. We will find an extension $r^*\leq r$ so that
    $$
    r^*\Vdash U\cap (M\cap\dot{A})\neq\es.
    $$

    Since $r\leq q$ and $q$ is $(M,\ps)$-strongly generic, let $\bar{r}:=\res_M(r)$. Let $A_{\bar{r}}$ be the trace of $\dot{A}$ below $\bar{r}$, namely
    $$
    A_{\bar{r}}:=\lb a\in X:(\exists t\leq_\ps\bar{r})\;t\Vdash\check{a}\in\dot{A}\rb;
    $$
    we emphasize that a witness to $a\in A_{\bar{r}}$ needn't be a condition in $M$, only a condition below $\bar{r}$.

    Arguing in the ground model $V$, we claim that $x^*\in\cl(A_{\bar{r}})$. Thus fix $E\in\tau$ with $x^*\in E$. Let $t\leq\bar{r}$ (not necessarily with $t\in M$) be any condition which forces that $\check{E}\cap\dot{A}\neq\es$. By further extending $t$ if necessary, we may find a specific $a\in X$ so that
    $$
    t\Vdash\check{a}\in\check{E}\cap\dot{A}.
    $$
    Then $a\in A_{\bar{r}}$, since $t$ is a witness, and moreover, $a\in E$. This completes the proof that $x^*\in\cl(A_{\bar{r}})$.

    Continuing, we observe that $A_{\bar{r}}\in M$, since $A_{\bar{r}}$ is definable in $H(\theta)$ by parameters in $M$. Since $(X,\tau)$ is countably tight in $V$, we may apply the elementarity of $M$ to find a $B\in\left[A_{\bar{r}}\right]^{\aleph_0}$ with $B\in M$ so that $x^*\in\cl(B)$. Since $B\in M$ and $B$ is countable, we in fact have that $B\seq M$.

    Now recall the open set $U$, fixed above. Since $x^*\in U\cap\cl(B)$, we may find a point $a\in B\cap U$. But $B\seq M$, so $a\in M$. Moreover, $B\seq A_{\bar{r}}$, so $a\in A_{\bar{r}}$ as well. By applying the elementarity of $M$ once again, we may find an extension $t\leq\bar{r}$ with $t\in M$ so that $t\Vdash\check{a}\in\dot{A}$. Hence
    $$
    t\Vdash\check{a}\in U\cap (M\cap\dot{A}).
    $$
    However, since $t\leq\bar{r}$ and since $t$ is in $M$, we know that $t$ is compatible with $r$. Hence we may fix an extension $r^*$ of $t$ and $r$. This $r^*$ then finishes the proof.
\end{proof}

\section{Preservation of some Covering Properties}

In this section we will prove that forcings which are strongly proper for stationarily-many countable elementary submodels preserve the Rothberger and Menger properties of a space. To do so, we will need to translate the idea of an ``endowment" (originally from \cite{DowRemote}, and later used in \cite{DTWNewProofs}, \cite{GJTForcingAndNormality}, and \cite{Kada}) to this context. The first subsection will handle this task, and in the second subsection, we will prove the preservation theorems.

\subsection{Endowments, Approximations, and Strongly Proper Forcings}\label{subsection:Endow}

For the whole of this subsection, we fix the following objects:
\begin{itemize}
    \item a poset $\ps$;
    \item a Lindel{\"o}f space $(X,\tau)$;
    \item a $\ps$-name $\dot{\cal{U}}$ for an open cover of $(X,\tau)$ forced to consist of elements of $\tau$;
    \item a large enough regular cardinal $\theta$;
    \item a countable $M\prec H(\theta)$ containing $\ps, (X,\tau)$, and $\dot{\cal{U}}$ as parameters; and
    \item an enumeration $\vp_M:\om\to M\cap\ps$ of $M\cap\ps$.
\end{itemize}

We will (slightly) generalize the notion of an endowment to the present context in order to establish the existence of useful approximations to generic open covers.

Before launching into the precise definitions, we take a moment to get the reader acquainted with the idea. Recall that a maximal antichain $A$ in $M\cap\ps$ is an antichain so that
$$
\text{for all }p\in M\cap\ps,\;\text{there is a }q\in A\text{ s.t. }q\text{ is compatible with }p.
$$
For $n\in\om$, an element $A_0$ of an ``$n$-endowment" $\cal{L}^M_n$ will be a \emph{finite} antichain in $M\cap\ps$ which contains witnesses to the above statement, but only applied to the first $(n+1)$-many conditions in $M$. Namely, $A_0$ will satisfy
$$
\text{for all }k\leq n,\;\text{there is a }q\in A_0\text{ s.t. }q\text{ is compatible with }\vp_M(k).
$$
The point of these finite antichains is that each condition $q$ in such an $A_0$ will come paired with an open set which $q$ forces into $\dot{\cal{U}}$, and (wishing to hedge our bets) we intend to intersect these finitely-many open sets.

Now for the more precise definitions. Let $\cal{A}_M$ denote the set of all maximal antichains of the poset $M\cap\ps$. For each $A\in\cal{A}_M$, let $A(n)$ be a finite subset of $A$ satisfying the following: for all $k\leq n$, there exists $q\in A(n)$ so that $\vp_M(k)$ is compatible with $q$. 

\begin{definition}
    We define the \emph{$n$-th endowment of $M$ and $\ps$}, denoted $\cal{L}^M_n$, to be the set
    $$
    \cal{L}^M_n:=\lb A(n):A\in\cal{A}_M\rb.
    $$

    The \emph{endowment of $M$ and $\ps$}, denoted $\cal{L}^M$, is the sequence $\la\cal{L}^M_n:n\in\om\ra$.
\end{definition}

Now we are ready to use $\cal{L}^M$ to define a sequence $\la\cal{V}^M_n(\dot{\cal{U}}):n\in\om\ra$ of approximations to $\dot{\cal{U}}$.

\begin{remark}\label{remark:ModelLindelof}
    We emphasize here that we only define the sequence $\la\cal{V}^M_n(\dot{\cal{U}}):n\in\om\ra$ of approximations to $\dot{\cal{U}}$ when $(X,\tau)$ is Lindel{\"o}f and $\dot{\cal{U}}$ a member of $M$.
\end{remark}

Fix $n\in\om$. For each $x\in X$, we may apply Lemma \ref{lemma:lovely} to find a maximal antichain $A_x\seq M\cap\ps$ so that for all $p\in A_x$ there is an open $W_{x,p}\in\tau\cap M$ so that
$$
p\Vdash x\in\check{W}_{x,p}\in\dot{\cal{U}}.
$$
Let
$$
V_{x,n}:=\bigcap\lb W_{x,p}:p\in A_x(n)\rb.
$$
Since $A_x(n)$ is finite, $V_{x,n}$ is open; also, $x\in V_{x,n}$. Finally, $V_{x,n}\in M$ since each $W_{x,p}$ is in $M$ and since $A_x(n)$ is finite.

\begin{definition}
    The \emph{$n$-th $M$-approximation to $\dot{\cal{U}}$}, denoted $\cal{V}^M_n(\dot{\cal{U}})$, is defined as
    $$
    \cal{V}^M_n(\dot{\cal{U}}):=\lb V_{x,n}:x\in X\rb.
    $$
\end{definition}

\begin{remark}
    Note that $\cal{V}^M_n(\dot{\cal{U}})$ is an open cover of $(X,\tau)$, since for each $x\in X$, $x$ is a member of $V_{x,n}$.
\end{remark}

The next lemma, though simple, will be quite useful:

\begin{lemma}\label{lemma:bacon}
    For each $n\in\om$, each $E\in\cal{V}^{M}_n(\dot{\cal{U}})$, and each $k\leq n$, there is a $d \leq_{M\cap\ps}\vp_M(k)$ so that
    $$
    d\Vdash (\exists U\in\dot{\cal{U}})\;\left[\check{E}\subseteq U\right].
    $$
\end{lemma}
\begin{proof}
    Fix $n,k$, and $E$. Let $x\in X$ so that $E=V_{x,n}$. As $k\leq n$, we may find $p\in A_x(n)$ so that $p$ is compatible with $\vp_M(k)$. By the elementarity of $M$, we may let $d$ be an extension of $p$ and $\vp_M(k)$ with $d\in M$.

    Then $E=V_{x,n}\seq W_{x,p}$. Moreover, $p$ forces that $x\in\check{W}_{x,p}\in\dot{\cal{U}}$. Since $d$ extends $p$ and $\vp_M(k)$, $d$ is the desired witness.
\end{proof}

\subsection{More Preservation Theorems}

We now show that posets which are strongly proper for stationarily many models preserve the Rothberger and Menger properties. The proofs are nearly identical. Accordingly, we will prove the preservation of the Rothberger property in detail, and then we will carry out the slight modifications sufficient to prove the preservation of the Menger property. We close with some corollaries concerning the preservation of $\om$-Rothberger and $\om$-Menger spaces.

\begin{theorem}\label{theorem:PreserveRothberger}
    Suppose that $\theta$ is a large enough regular cardinal, that $(X,\tau)\in H(\theta)$ is a Rothberger space, and that $\ps\in H(\theta)$ is a poset. Suppose further that the set of countable $M\prec H(\theta)$ so that $\ps$ is strongly proper for $M$ is stationary in $[H(\theta)]^{\aleph_0}$. Then $\ps$ forces that $(X,\tau)$ is Rothberger.
\end{theorem}
\begin{proof}
    Fix a condition $p\in\ps$ and a sequence $\la\dot{\cal{U}}_n:n\in\om\ra$ of names for open covers of $(X,\tau)$. We may assume that each $\,\dot{\cal{U}}_n$ is forced to consist of basic open sets, and hence is forced to be a subset of $\tau$. Our goal is to find an extension $q\leq p$ and sequence $\la\dot{U}_n:n\in\om\ra$ of names so that
    \begin{enumerate}
        \item for each $n\in\om$, $q\Vdash\dot{U}_n\in\dot{\cal{U}}_n$ and
        \item $q$ forces $\lb\dot{U}_n:n\in\om\rb$ is an open cover of $X$.
    \end{enumerate}

    Using the stationarity assumed in the statement of the theorem, let $M\prec H(\theta)$ be countable so that both of the following are true:
    \begin{itemize}
        \item $M$ contains the parameters $p$, $\ps$, $(X,\tau)$, and $\la\dot{\cal{U}}_n:n\in\om\ra$;
        \item every condition in $M$ extends to an $(M,\ps)$-strongly generic condition.
    \end{itemize}
    In particular, we may extend $p$ to a condition $q$ which is $(M,\ps)$-strongly generic.

    Let $\vp_M:\om\to M\cap\ps$ be the enumeration of $M\cap\ps$ relative to which we construct the approximations as in Subsection \ref{subsection:Endow}.

    Since each $\dot{\cal{U}}_n$ is a member of $M$ and since $(X,\tau)$ is Lindel{\"o}f, we know that $\cal{V}^M_n(\dot{\cal{U}}_n)$ is an open cover of $(X,\tau)$ for each $n\in\om$ (recall Remark \ref{remark:ModelLindelof}). We now apply the fact that $(X,\tau)$ is Rothberger in $V$. By Lemma \ref{lemma:StrongBerger} applied to the sequence $\la\cal{V}^M_n(\dot{\cal{U}}_n):n\in\om\ra$, we may find a sequence $\la W_n:n\in\om\ra$ so that $W_n\in\cal{V}^M_n(\dot{\cal{U}}_n)$ for each $n$ and so that for each $x\in X$, there are \emph{infinitely}-many $n$ with $x\in W_n$.

    Now we will define a sequence $\la\dot{U}_n:n\in\om\ra$ of names and show that $q$ forces that this sequence witnesses the definition of Rothberger in the extension. We apply the maximality principle for names by showing, given an arbitrary $V$-generic filter $G$ over $\ps$, how to define $\la\dot{U}_n[G]:n\in\om\ra$. So fix $G$, and fix $n$. If there is some $U\in\cal{U}_n$ so that $W_n\seq U$, then we let $\dot{U}_n[G]$ select some such $U$. Otherwise, we let $\dot{U}_n[G]\in\cal{U}_n$ be arbitrary.

    We show that $q$ forces that $\la\dot{U}_n:n\in\om\ra$ is an open cover of $X$ by showing that no extension of $q$ can force a specific counterexample. So fix $r\leq q$ and $x\in X$, and we will find an $r^*\leq r$ and an $n$ so that $r^*\Vdash x\in\dot{U}_n$.

    Consider $\bar{r}$, a residue of $r$ to $M\cap\ps$. Choose $k$ so that $\vp_M(k)=\bar{r}$, where $\vp_M:\om\to M\cap\ps$ is the enumeration that we fixed earlier in the proof. Since $\la W_n:n\in\om\ra$ satisfies Lemma \ref{lemma:StrongBerger}, we may find an $m\geq k$ so that $x\in W_m$. We now apply Lemma \ref{lemma:bacon} to get a condition $d\leq\bar{r}$ with $d\in M\cap\ps$ so that
    $$
    d\Vdash(\exists U\in\dot{\cal{U}}_m)\;[\check{W}_m\seq U].
    $$
    Since $d\leq_{M\cap\ps}\bar{r}$ and $\bar{r}$ is a residue of $r$ to $M$, we may find a condition $r^*\leq r,d$. Then since $r^*$ extends $d$,
    $$
    r^*\Vdash(\exists U\in\dot{\cal{U}}_m)\;[\check{W}_m\seq U].
    $$
    By definition of the name $\dot{U}_m$,  we then have
    $$
    r^*\Vdash x\in \check{W}_m\seq\dot{U}_m\in\dot{\cal{U}}_m.
    $$
    This completes the proof.
\end{proof}

As mentioned at the beginning of this subsection, the following proof is nearly identical to the proof of Theorem \ref{theorem:PreserveRothberger}.

\begin{theorem}
    Suppose that $\theta$ is a large enough regular cardinal, that $(X,\tau)\in H(\theta)$ is a Menger space, and that $\ps\in H(\theta)$ is a poset. Suppose further that the set of countable $M\prec H(\theta)$ so that $\ps$ is strongly proper for $M$ is stationary in $[H(\theta)]^{\aleph_0}$. Then $\ps$ forces that $(X,\tau)$ is Menger.
\end{theorem}
\begin{proof}
    Fix a condition $p\in\ps$ and a sequence $\la\dot{\cal{U}}_n:n\in\om\ra$ of names for open covers of $(X,\tau)$. We may assume that each $\,\dot{\cal{U}}_n$ is forced to consist of basic open sets, and hence is forced to be a subset of $\tau$. Our goal is to find an extension $q\leq p$ and sequence $\la\dot{\cal{F}}_n:n\in\om\ra$ of $\ps$-names so that
    \begin{enumerate}
        \item for each $n\in\om$, $q\Vdash\dot{\cal{F}}_n\text{ is a finite subset of }\dot{\cal{U}}_n$; and
        \item $q\Vdash\bigcup\lb\dot{\cal{F}}_n:n\in\om\rb$ is an open cover of $X$.
    \end{enumerate}

    Using the stationarity assumed in the statement of the theorem, let $M\prec H(\theta)$ be countable so that both of the following are true:
    \begin{itemize}
        \item $M$ contains the parameters $p$, $\ps$, $(X,\tau)$, and $\la\dot{\cal{U}}_n:n\in\om\ra$;
        \item every condition in $M$ extends to an $(M,\ps)$-strongly generic condition.
    \end{itemize}
    In particular, we may extend $p$ to a condition $q$ which is $(M,\ps)$-strongly generic.

    Let $\vp_M:\om\to M\cap\ps$ be the enumeration of $M\cap\ps$ relative to which we construct the approximations as in Subsection \ref{subsection:Endow}.

    Apply Lemma \ref{lemma:StrongMenger} to the sequence $\la\cal{V}^M_n(\dot{\cal{U}}_n):n\in\om\ra$ to find a sequence $\la\cal{F}^M_n:n\in\om\ra$ so that each $\cal{F}^M_n$ is a finite subset of $\cal{V}^M_n(\dot{\cal{U}}_n)$, so that $\bigcup\lb\cal{F}^M_n:n\in\om\rb$ is an open cover of $(X,\tau)$, and so that for each $x\in X$, there are infinitely-many $k$ so that $x\in\bigcup\cal{F}^M_k$.

    We now define the sequence $\la\dot{\cal{F}}_n:n\in\om\ra$ of $\ps$-names which will witness the desired statement. Accordingly, let $G$ be an arbitrary $V$-generic filter over $\ps$, and fix $n\in\om$. We have two cases. First suppose that no $F\in\cal{F}^M_n$ is a subset of any $U\in\mathcal{U}_n$. In this degenerate case, we let $\cal{F}_n$ be an arbitrary finite subset of $\cal{U}_n$. Now suppose that at least one $F\in\cal{F}^M_n$ is a subset of an element of $\mathcal{U}_n$. In this case, for each such $F\in\cal{F}^M_n$, choose $U_F\in\mathcal{U}_n$ with $F\subset U_F$, and let $\cal{F}_n$ be the set of these $U_F$'s.
    

    The proof that $q$ forces that $\bigcup\lb\dot{\cal{F}}_n:n\in\om\rb$ is an open cover of $X$ is nearly identical to the corresponding proof in Theorem \ref{theorem:PreserveRothberger}.
\end{proof}

These theorems combined show that strongly proper forcings preserve the $\omega$-Rothberger and $\omega$-Menger properties as well.

\begin{corollary}
    Suppose that $\theta$ is a large enough regular cardinal, that $(X,\tau)\in H(\theta)$ is an $\omega$-Rothberger (resp. $\om$-Menger) space, and that $\ps\in H(\theta)$ is a poset. Suppose further that the set of countable $M\prec H(\theta)$ so that $\ps$ is strongly proper for $M$ is stationary in $[H(\theta)]^{\aleph_0}$. Then $\ps$ forces that $(X,\tau)$ is $\omega$-Rothberger (resp. $\om$-Menger).    
\end{corollary}
\begin{proof}
    As in Corollary \ref{OmegaLindelof}, this follows from the fact that $X$ is $\omega$-Rothberger (resp. $\om$-Menger) if and only if for all $n$, $X^n$ is Rothberger (resp. Menger).
\end{proof}

\section{Selection Games}

Selection principles have naturally corresponding selection games, which include types of topological games.
Topological games have a long history, much of which can be gathered from Telg{\'a}rsky's survey \cite{Telgarsky}.
In this paper, we consider the traditional selection games for two players, $\rm{I}$ and $\rm{II}$, of countably infinite length.
\begin{definition}
    Given sets $\mathcal A$ and $\mathcal B$, we define the \emph{finite-selection game}
    $\mathsf{G}_{\mathrm{fin}}(\mathcal A, \mathcal B)$ for $\mathcal A$ and $\mathcal B$ as follows.
    In round $n \in \omega$, $\rm{I}$ plays $A_n \in \mathcal A$ and $\rm{II}$ responds with $\mathcal F_n \in [A_n]^{<\omega}$.
    We declare $\rm{II}$ the winner if $\bigcup\{ \mathcal F_n : n \in \omega \} \in \mathcal B$.
    Otherwise, $\rm{I}$ wins.
\end{definition}
\begin{definition}
    Given sets $\mathcal A$ and $\mathcal B$, we analogously define the \emph{single-selection game}
    $\mathsf{G}_{1}(\mathcal A, \mathcal B)$ for $\mathcal A$ and $\mathcal B$ as follows.
    In round $n \in \omega$, $\rm{I}$ plays $A_n \in \mathcal A$ and $\rm{II}$ responds with $x_n \in A_n$.
    We declare $\rm{II}$ the winner if $\{x_n : n \in \omega \} \in \mathcal B$.
    Otherwise, $\rm{I}$ wins.
\end{definition}
\begin{definition}
    By \emph{selection games}, we mean the class consisting of $\mathsf G_\square(\mathcal A, \mathcal B)$
    where $\square \in \{1, \mathrm{fin} \}$, and $\mathcal A$ and $\mathcal B$ are sets.
    So, when we say $\mathcal G$ is a selection game, we mean that there exist $\square \in \{1 , \mathrm{fin} \}$
    and sets $\mathcal A, \mathcal B$ so that $\mathcal G = \mathsf G_\square(\mathcal A, \mathcal B)$.
\end{definition}
The study of games naturally inspires questions about the existence of various kinds of strategies.
Infinite games and corresponding full-information strategies were both introduced in \cite{GaleStewart}.
Some forms of limited-information strategies came shortly after, like positional (also known as stationary) strategies.
For more on stationary and Markov strategies, see \cite{GalvinTelgarsky}.
\begin{definition}
    We define strategies of various strengths below.
    \begin{itemize}
    \item
    A \emph{strategy for $\rm{I}$} in $\mathsf{G}_1(\mathcal A, \mathcal B)$ is a function
    $\sigma:(\bigcup \mathcal A)^{<\omega} \to \mathcal A$.
    A strategy $\sigma$ for $\rm{I}$ is called \emph{winning} if whenever $x_n \in \sigma\langle x_k : k < n \rangle$
    for all $n \in \omega$, $\{x_n: n\in\omega\} \not\in \mathcal B$.
    If $\rm{I}$ has a winning strategy, we write $\mathrm{I} \uparrow \mathsf{G}_1(\mathcal A, \mathcal B)$.
    \item
    A \emph{strategy for $\rm{II}$} in $\mathsf{G}_1(\mathcal A, \mathcal B)$ is a function
    $\sigma:\mathcal A^{<\omega} \to \bigcup \mathcal A$.
    A strategy $\sigma$ for $\rm{II}$ is \emph{winning} if whenever $A_n \in \mathcal A$ for all $n \in \omega$,
    $\{\sigma(A_0,\ldots,A_n) : n \in \omega\} \in \mathcal B$.
    If $\rm{II}$ has a winning strategy, we write $\mathrm{II} \uparrow \mathsf{G}_1(\mathcal A, \mathcal B)$.
    \item
    A \emph{predetermined strategy} for $\rm{I}$ is a strategy which only considers the current turn number.
    Formally it is a function $\sigma: \omega \to \mathcal A$.
    If $\rm{I}$ has a winning predetermined strategy, we write $\mathrm{I} \underset{\mathrm{pre}}{\uparrow} \mathsf{G}_1(\mathcal A, \mathcal B)$.
    \item
    A \emph{Markov strategy} for $\rm{II}$ is a strategy which only considers the most recent move of $\rm{I}$ and the current turn number.
    Formally it is a function $\sigma :\mathcal A \times \omega \to \bigcup \mathcal A$.
    If $\rm{II}$ has a winning Markov strategy, we write
    $\mathrm{II} \underset{\mathrm{mark}}{\uparrow} \mathsf{G}_1(\mathcal A, \mathcal B)$.
    \item
    If there is a single element $A_0 \in \mathcal A$ so that the constant function with value $A_0$ is a winning strategy for $\rm{I}$, we say that $\rm{I}$ has a \emph{constant winning strategy}, denoted by $\mathrm{I} \underset{\mathrm{cnst}}{\uparrow} \mathsf{G}_1(\mathcal A, \mathcal B)$.
    \end{itemize}
    These definitions can be extended to $\mathsf{G}_{\mathrm{fin}}(\mathcal A, \mathcal B)$ in the obvious way.
\end{definition}
Note that, for any selection game $\mathcal G$,
\[
    \mathrm{II} \underset{\mathrm{mark}}{\uparrow} \mathcal G
    \implies \mathrm{II} \uparrow \mathcal G
    \implies \mathrm{I} \not\uparrow \mathcal G
    \implies \mathrm{I} \underset{\mathrm{pre}}{\not\uparrow} \mathcal G
    \implies \mathrm{I} \underset{\mathrm{cnst}}{\not\uparrow} \mathcal G.
\]

Strategies for selection games are connected to selection principles. Notably, $\mathrm{I} \underset{\mathrm{pre}}{\not\uparrow} \mathcal G_1(\mathcal A, \mathcal B)$ if and only if $S_1(\mathcal A, \mathcal B)$ is true. As another example, $\rm{I}$ fails to have a constant winning strategy in the Rothberger game ($G_1(\mathcal O(X), \mathcal O(X))$) if and only if the space is Lindel\"{o}f.

For the Rothberger game, some of the levels of strategies are equivalent. Pawlikowski showed that $\rm{I}$ has a winning strategy in the Rothberger game if and only if $\rm{I}$ has a pre-determined winning strategy \cite{Pawlikowski}. Scheepers (\cite{Scheepers1997}) later extended this result to the $\omega$-Rothberger game, $G_1(\Omega(X),\Omega(X))$.

In \cite{CCH} it is shown that $\rm{II}$ has a winning strategy in the Rothberger game if and only if $\rm{II}$ has a winning strategy in the $\omega$-Rothberger game.

Finally, its straightforward to check that the following are equivalent:
\begin{itemize}
    \item $\rm{II}$ has a winning Markov strategy in the Rothberger game on $X$,
    \item $\rm{II}$ has a winning Markov strategy in the $\omega$-Rothberger game on $X$, and
    \item $X$ is topologically countable (if $X$ is $T_1$ this is the same as saying $X$ is countable; for the full definition of topologically countable, see \cite{CCH}).
\end{itemize}

The following counterexamples distinguish the rest of the properties from each other.
\begin{itemize}
    \item $\rm{II}$ has a winning strategy in the Rothberger game on the Fortissimo space of the reals, but player two does not have a winning Markov strategy, (see the following stack exchange post \cite{CaruvanaStack}).
    \item Neither $\rm{I}$ nor $\rm{II}$ has a winning strategy in the Rothberger game on Luzin subsets of the real line \cite{Reclaw}.
\end{itemize}

\section{A Preservation Theorem for the Rothberger Game}

In this section we prove that strongly proper forcings preserve spaces in which player ${\rm II}$ has a winning strategy in the Rothberger game. We begin by proving an ad hoc density lemma, and then we prove the preservation theorem.

\begin{lemma}\label{lemma:DensityRothberger}
    Let $\theta$ be a large enough regular cardinal, $\ps\in H(\theta)$ a poset, and $(X,\tau)\in H(\theta)$ a space so that $\rm{II}$ has a winning strategy, say $\si$, in the Rothberger game on $(X,\tau)$. Assume that the set of countable $M\prec H(\theta)$ so that $\ps$ is strongly proper for $M$ is stationary in $[H(\theta)]^{\aleph_0}$.
    
    Suppose that $\cal{V}_1,\dots,\cal{V}_k$ are a sequence of open covers of $(X,\tau)$ played by $\rm{I}$ with $W_1,\dots,W_k$ the responses by $\rm{II}$ via $\si$. Finally, let $\,\dot{\cal{U}}\in H(\theta)$ be a $\ps$-name for an open cover of $X$ forced to consist of sets in $\tau$. Then the set of $q\in\ps$ with the following property $(*)$ is dense: $(*)$ asserts of $q$ that there exists a countable $M\prec H(\theta)$ with $\lb \dot{\cal{U}},\ps, (X,\tau)\rb\subset M$ so that 
    \begin{enumerate}
        \item $q$ is a strongly $(M,\ps)$-generic condition and
        \item there exists an $n\geq 1$ so that
    $$
    q\Vdash (\exists U\in\dot{\cal{U}})\;\left[\si\left(\cal{V}_1,\dots,\cal{V}_k,\cal{V}^M_n(\dot{\cal{U}})\right)\seq U \right].
    $$
    \end{enumerate}
\end{lemma}
\begin{proof}
    Let $p_0$ be a given condition in $\ps$. Using the stationarity assumed for the lemma, let $M\prec H(\theta)$ be countable containing $p_0$, $\dot{\cal{U}}$, $(X,\tau)$, and $\ps$ as elements. Let $\vp_M:\om\to M\cap\ps$ be the enumeration of $M\cap\ps$ relative to which we construct the endowments. Extend $p_0$ to an $(M,\ps)$-strongly generic condition $p$, and let $n\in\om$ so that 
    $$
    \vp_M(n)=\res_M(p).
    $$

    Let $W_{k+1}$ be $\rm{II}$'s reply via $\si$ where $\rm{I}$ plays $\cal{V}_1,\dots,\cal{V}_k,\cal{V}^M_n(\dot{\cal{U}})$. Among other things, $W_{k+1}\in \cal{V}^M_n(\dot{\cal{U}})$. Apply Lemma \ref{lemma:bacon} to find a condition $t\leq _{M\cap\ps}\res_M(p)$ so that
    $$
    t\Vdash\left(\exists U\in\dot{\cal{U}}\right)\left[\check{W}_{k+1}\seq U\right].
    $$
    Since $t\in M$ and $t\leq\res_M(p)$, we can find a condition $q$ extending $p$ and $t$. Then $q$ is $(M,\ps)$-strongly generic since it extends $p$, and $q$ forces the desired statement since it extends $t$.
\end{proof}

Now we are ready for the next preservation theorem.

\begin{theorem}
    Suppose that $\theta$ is a large enough regular cardinal, that $(X,\tau)\in H(\theta)$ is a space for which II has a winning strategy in the Rothberger game on $X$, and that $\ps\in H(\theta)$ is a poset. Suppose further that the set of countable $M\prec H(\theta)$ so that $\ps$ is strongly proper for $M$ is stationary in $[H(\theta)]^{\aleph_0}$. Then $\ps$ forces that II has a winning strategy in the Rothberger game on $(X,\tau)$.
\end{theorem}
\begin{proof}
    Fix the space $(X,\tau)$ and a winning strategy $\si$ for player ${\rm II}$ in the Rothberger game on $(X,\tau)$.

    Let $G$ be an arbitrary $V$-generic filter over $\ps$, and in $V[G]$, we will recursively construct a strategy for player ${\rm II}$ in the Rothberger game on $(X,\tau_G)$. This strategy will be denoted $\si^*$.

    The main idea of the proof is this: in $V[G]$, we will simulate a run of the ground model game for $(X,\tau)$ where ${\rm II}$ plays via $\si$. ${\rm I}$'s plays in this game will come from approximations of the form $\cal{V}^{M_n}_{\ell(n)}(\dot{\cal{U}}_n)$ for models $M_n$, and where $\dot{\cal{U}}_n$ is a name for ${\rm I}$'s play in the game for $(X,\tau_G)$. We emphasize here that we will be picking objects from $V$ while working in the extension $V[G]$. To see that this approach succeeds, we will step back to $V$ and work with names for the relevant objects. We will close out the argument by simulating a run of a game in $V$ and catching our tail.

    To begin, suppose that in the Rothberger game on $(X,\tau_G)$, ${\rm I}$ has played $\cal{U}_1,\dots,\cal{U}_m$ and that ${\rm II}$ has responded via $\si^*$ with $U_1,\dots,U_m$. Let ${\rm I}$ now play $\cal{U}_{m+1}$. The situation looks like this:
   $$
   (X,\tau_G)\text{-game}:
   \begin{array}{c|cccccccc}
         {\rm I} & \cal{U}_1 && \dots &&\cal{U}_m &&\cal{U}_{m+1}\\
         \hline
         {\rm II} & & U_1  &&\dots && U_m&
   \end{array}
   $$
We will work to define ${\rm II}$'s reply via $\si^*$, denoted $U_{m+1}$. To do so, we first articulate a recursive hypothesis, and then we will construct $U_{m+1}$ by applying Lemma \ref{lemma:DensityRothberger}.

We suppose as a recursive hypothesis that we have defined a sequence of four-tuples $\la M_k,\dot{\cal{U}}'_k,p_k,\ell(k)\ra_{1\leq k\leq m}$ consisting of elements of $V$ that satisfy six assumptions; we will state the first five of these, make a comment, and then state the sixth. The first five are:
\begin{enumerate}
    \item $M_k$ is a countable elementary submodel of $H^V(\theta)$;
    \item $\dot{\cal{U}}_k\in M_k$ is a $\ps$-name for an open cover of $(X,\tau)$ consisting of elements of $\tau$;
    \item $\dot{\cal{U}}'_k[G]=\cal{U}_k$;
    \item $p_k\in G$ and $p_k$ is an $(M_k,\ps)$-strongly generic condition;
    \item $\res_{M_k}(p_k)$ is the $\ell(k)$-th condition in the enumeration of $M_k\cap\ps$, i.e., $\vp_{M_k}(\ell(k))=\res_{M_k}(p_k)$.
\end{enumerate}
Next, let $W_1,\dots,W_m$ be ${\rm II}$'s replies via $\si$ in the partial run of the game in $V$, where ${\rm I}$ plays $\la\cal{V}^{M_k}_{\ell(k)}(\dot{\cal{U}}'_k):1\leq k\leq m\ra$. So the situation, for the ground model game being played in $V[G]$, looks like this:
$$
   (X,\tau)\text{-game}:
   \begin{array}{c|ccccccc}
         {\rm I} & \cal{V}^{M_1}_{\ell(1)}(\dot{\cal{U}}'_1) && \dots &&\cal{V}^{M_m}_{\ell(m)}(\dot{\cal{U}}'_m)&\\
         \hline
         {\rm II} & & W_1  &&\dots && W_m
   \end{array}
   $$
Our sixth and final recursive assumption is
\begin{enumerate}
    \item[(6)] for each $1\leq k\leq m$,
    $$
    p_k\Vdash^V_\ps\,\left(\exists U\in\dot{\cal{U}}'_k\right)\left[\check{W}_k\seq U\right].
    $$
\end{enumerate}

Now we apply Lemma \ref{lemma:DensityRothberger} to the sequence $\la \cal{V}^{M_k}_{\ell(k)}(\dot{\cal{U}}'_k):1\leq k\leq m\ra$ of ${\rm I}$'s plays and to a name $\dot{\cal{U}}'_{m+1}$ for $\cal{U}_{m+1}$ to find a condition $p_{m+1}\in G$ and a countable elementary submodel $M_{m+1}\prec H^V(\theta)$ (with $M_{m+1}$ in the ground model) so that
\begin{itemize}
    \item $p_{m+1}$ is $(M_{m+1},\ps)$-strongly generic;
    \item $M_{m+1}$ contains $\ps$, $(X,\tau)$, and $\dot{\cal{U}}'_{m+1}$;
    \item $p_{m+1}$ forces over $V$ that $\left(\exists U\in\dot{\cal{U}}'_{m+1}\right)\left[\check{W}_{m+1}\seq U\right]$,
\end{itemize}
where $W_{m+1}$ is ${\rm II}$'s response via $\si$. We may finally let ${\rm II}$'s reply via $\si^*$ in the $(X,\tau_G)$-game be some $U\in\cal{U}_{m+1}$ with $W_{m+1}\seq U$.

To complete the definition of $\si^*$, we pay our debt to the recursion hypothesis by letting the next four-tuple be equal to $\la M_{m+1},\dot{\cal{U}}'_{m+1},p_{m+1},\ell(m+1)\ra$, where $$
\vp_{M_{m+1}}(\ell(m+1))=\res_{M_{m+1}}(p_{m+1}).
$$
This completes the definition of $\si^*$ in an arbitrary $V$-generic extension via $\ps$ and hence completes the definition of the name $\dot{\si}^*$.

Now, working over $V$, we show that $\dot{\si}^*$ is forced to be a winning strategy for ${\rm II}$ in the extension. Towards that end, fix sequences $\la\dot{\cal{U}}_k:k\geq 1\ra$ and $\la\dot{U}_k:k\geq 1\ra$ which are forced to be a run of the game in which ${\rm II}$ plays via $\dot{\si}^*$. By definition of $\dot{\si}^*$ we may also find a sequence $\la\dot{W}_k:k\geq 1\ra$ and a sequence $\la\dot{M}_k,\ddot{\cal{U}}_k,\dot{p}_k,\dot{\ell}(k)\ra_{k\geq 1}$ of $\ps$-names for elements of $V$ forced to satisfy items (1)-(6) above. Among other things, each $\dot{p}_k$ is forced to be in $\dot{G}$ and $\ddot{\cal{U}}_k$ (a name for a name) is forced to be some element $z_k$ of $\dot{M}_k$ so that the interpretation of $z_k$ equals the interpretation of $\dot{\cal{U}}_k$.

We will show that $\ps$ forces that $\lb\dot{U}_k:k\geq 1\rb$ is an open cover and hence that ${\rm II}$ wins. To do so, fix a condition $p\in\ps$ and a point $x\in X$, and we will show that some extension of $p$ forces that $x$ is in $\bigcup\lb\dot{U}_k:k\geq 1\rb$.

To find the desired extension of $p$, we create a run of the Rothberger game in $V$ where ${\rm II}$ plays according to $\si$. This will involve constructing an $\om$-length decreasing sequence $\la r_n:n\in\om\ra$ of conditions starting from $p$, each of which reveals more about the sequence of names.

More precisely, define $\la r_n:n\in\om\ra$ with $r_0=p$ so that
\begin{enumerate}
    \item[(a)] for all $n\geq 1$, $r_n$ decides the value of $\la\dot{M}_n,\ddot{\cal{U}}_n,\dot{p}_n,\dot{\ell}(n)\ra$, say as 
    $$
    \la M_n,\dot{\cal{U}}'_n,p_n,\ell(n)\ra;
    $$
    \item[(b)] $r_n\leq p_n$ (this can be done since $r_n\Vdash\check{p}_n=\dot{p}_n\in\dot{G}$);
    \item[(c)] $r_n\Vdash\dot{\cal{U}}'_n=\dot{\cal{U}}_n$.
\end{enumerate}
Since $r_n$ forces that the value of the tuple $\la\dot{M}_n,\ddot{\cal{U}}_n,\dot{p}_n,\dot{\ell}(n)\ra$ is equal to 
$$
\la M_n,\dot{\cal{U}}'_n,p_n,\ell(n)\ra,
$$
we know that $\dot{\cal{U}}'_n\in M_n$. Hence $\cal{V}^{M_n}_{\ell(n)}(\dot{\cal{U}}'_n)$ is defined and is an open cover of $(X,\tau)$. Then $\la\cal{V}^{M_n}_{\ell(n)}(\dot{\cal{U}}'_n):n\geq 1\ra$ is a valid sequence of moves for ${\rm I}$ in the Rothberger game on $(X,\tau)$ in the ground model. Let $\la B_n:n\geq 1\ra$ be ${\rm II}$'s replies via $\si$.

Since $\si$ is a winning strategy for player ${\rm II}$, we may find an $n\geq 1$ so that $x\in B_n$. We claim that $r_n\Vdash x\in\dot{U}_n$, which will finish the proof. Note that since the sequence of $r_i$'s is decreasing, $r_n$ decides the values of the previous tuples $\la\dot{M}_k,\ddot{\cal{U}}_k,\dot{p}_k,\dot{\ell}(k)\ra_{1\leq k<n}$ the same way that the previous $r_k$ do. Hence $r_n$ forces that 
$$
\dot{W}_n=\si\left(\cal{V}^{M_1}_{\ell(1)}(\dot{\cal{U}}'_1),\dots, \cal{V}^{M_n}_{\ell(n)}(\dot{\cal{U}}'_n)\right).
$$
Therefore $r_n\Vdash\dot{W}_n=\check{B}_n$. Since $r_n\leq p_n$, $r_n$ also forces that there is some $U\in\dot{\cal{U}}'_n$ so that $\dot{W}_n\seq U$. But then, as $r_n$ forces $\dot{\cal{U}}'_n=\dot{\cal{U}}_n$ and forces that $\dot{W}_n=\check{B}_n$, we conclude that $r_n$ forces that there is some $U\in\dot{\cal{U}}_n$ with $\check{B}_n\seq U$. Since $x\in B_n$, this completes the proof.
\end{proof}

We wrap up this section with some remarks on the preservation of Markov strategies for $\rm{II}$ in the Rothberger game. Recall that if a space $X$ is T1, then $\rm{II}$ has a Markov strategy in the Rothberger game if and only if $X$ is countable. From this we can already conclude that all forcing preserves Markov strategies for $\rm{II}$ in the Rothberger game for T1 spaces. As strongly proper forcings preserve $\om_1$, and hence cannot collapse uncountable objects to become countable, we also know that strongly proper forcings preserve spaces in which $\rm{II}$ does \emph{not} have a Markov strategy in the Rothberger game on T1 spaces. 

Even when the space is not T1, strongly proper forcings nonetheless preserve that II has a Markov strategy in the Rothberger game. The relevant proofs are almost identical to the proofs of the previous two results.

\begin{lemma}\label{LemmaDensityMarkovRothberger}
        Let $\theta$ be a large enough regular cardinal, $\ps\in H(\theta)$ a poset, and $(X,\tau)\in H(\theta)$ a space so that $\rm{II}$ has a winning Markov strategy, say $\si$, in the Rothberger game. 
        
        Suppose that $\cal{V}$ is an open cover of $(X,\tau)$ played by $\rm{I}$ at round $k$ of the game, and let $\dot{\cal{U}}\in H(\theta)$ be a $\ps$-name for an open cover of $X$ forced to consist of sets in $\tau$. Then the set of $q\in\ps$ with the following property $(*)$ is dense: $(*)$ asserts of $q$ that there exists a countable $M\prec H(\theta)$ with $\lb \dot{\cal{U}},\ps, (X,\tau)\rb\subset M$ so that 
    \begin{enumerate}
        \item $q$ is a strongly $(M,\ps)$-generic condition and
        \item there exists an $n\geq 1$ so that
    $$
    q\Vdash (\exists U\in\dot{\cal{U}})\;\left[\si\left(\cal{V},k\right)\seq U \right].
    $$
    \end{enumerate}
\end{lemma}

\begin{theorem}
    Suppose that $\theta$ is a large enough regular cardinal, that $(X,\tau)\in H(\theta)$ is so that II has a winning Markov strategy in the Rothberger game on $X$, and that $\ps\in H(\theta)$ is a poset. Suppose further that the set of countable $M\prec H(\theta)$ so that $\ps$ is strongly proper for $M$ is stationary in $[H(\theta)]^{\aleph_0}$. Then $\ps$ forces that II has a winning Markov strategy in the Rothberger game on $(X,\tau)$.
\end{theorem}

We conclude with a remark about the preservation of winning strategies for II in the $\omega$-Rothberger game. Recall that II has a winning strategy in the $\omega$-Rothberger game if and only if they have a winning strategy in the Rothberger game. The same is true for Markov strategies. We thus obtain the following corollary.

\begin{corollary}
    Suppose that $\theta$ is a large enough regular cardinal, that $(X,\tau)\in H(\theta)$ is so that II has a winning strategy (resp. Markov strategy) in the $\omega$-Rothberger game on $X$, and that $\ps\in H(\theta)$ is a poset. Suppose further that the set of countable $M\prec H(\theta)$ so that $\ps$ is strongly proper for $M$ is stationary in $[H(\theta)]^{\aleph_0}$. Then $\ps$ forces that II has a winning strategy (resp. Markov strategy) in the $\omega$-Rothberger game on $(X,\tau)$.
\end{corollary}

\section{Open Questions}

We have shown that strongly proper forcing preserves countable tightness and the Menger property, two selection principles, but in contrast to our results about the Rothberger game, we do not yet have a proof that strongly proper forcing preserves strategic information in the countable fan tightness or Menger games.

\begin{question}
    Suppose that $\ps$ is strongly proper for stationarily many models and that ${\rm II}$ has a winning strategy in the strong countable fan tightness game on a space $(X,\tau)$. Does $\ps$ preserve this?
\end{question}

\begin{question}
    Suppose that $\ps$ is strongly proper for stationarily many models and that ${\rm II}$ has a winning strategy in the Menger game on a space $(X,\tau)$. Does $\ps$ preserve this?
\end{question}

Even answering the following more limited question about the Menger game would be interesting.

\begin{question}
    Suppose that ${\rm II}$ has a winning strategy in the Menger game on a space $(X,\tau)$. Does Cohen forcing preserve this?
\end{question}

In the situation that the questions regarding the countable fan tightness game prove intractable, there is an easier question we can pose that relies on the connection between $\omega$-covers on $X$ and countable fan tightness for $C_p(X)$.

\begin{question}
    Suppose that $\ps$ is strongly proper for stationarily many models, $X$ is a topological space, and that ${\rm II}$ has a winning strategy in the strong countable fan tightness game on $(C_p(X),\tau)$. Does $\ps$ preserve this?
\end{question}

It is also common to study covering properties involving $k$-covers: $\mathcal U$ is a $k$-cover of $X$ if for every compact $K \subseteq X$ there is a $U \in \mathcal U$ so that $K \subseteq U$. The set of $k$-covers of $X$ is denoted $\mathcal K(X)$, and from this we obtain the $k$-Rothberger/Menger game $\mathsf{G}_\square(\mathcal A, \mathcal B)$. It is frequently the case that statements which are true for $\omega$-covers have analogous true statements for $k$-covers, so we ask the following question.

\begin{question}
    Suppose that $\ps$ is strongly proper for stationarily many models and that ${\rm II}$ has a winning strategy in the $k$-Rothberger/Menger game on a space $(X,\tau)$. Does $\ps$ preserve this?
\end{question}

\section*{Acknowledgements} We would like to thank Pedro Marun for providing a number of very helpful comments and corrections.

\end{document}